\documentclass{article}
\usepackage[english]{babel}

\usepackage{enumerate}
\usepackage[shortlabels]{enumitem}
\usepackage[T1]{fontenc}
\usepackage{graphicx,amsmath,amssymb,amsthm,enumitem,tikz,footnote}
\usepackage{thm-restate}
\usepackage{hyperref}       
\usepackage{url}            
\usepackage{booktabs}       
\usepackage{amsfonts}       
\usepackage{nicefrac}       

\usepackage[letterpaper,top=2.5cm,bottom=2.5cm,left=3cm,right=3cm,marginparwidth=1.75cm]{geometry}

\usepackage{hyperref}

\newtheorem{lemma}{Lemma}
\newtheorem{definition}{Definition}
\newtheorem{theorem}{Theorem}
\newtheorem{corollary}{Corollary}
\newtheorem{proposition}{Proposition}

\title{Polyharmonic Functions in the Quarter Plane}

\author{Andreas Nessmann}
\date{}

\begin{document}
\maketitle

\begin{abstract}
\noindent
In this article, a novel method to compute all discrete polyharmonic functions in the quarter plane for models with small steps, zero drift and a finite group is proposed. A similar method is then introduced for continuous polyharmonic functions, and convergence between the discrete and continuous cases is shown. \end{abstract}

\begin{small}
\noindent
This extended abstract was published in the proceedings of the 33\textsuperscript{rd} International Conference on Probabilistic, Combinatorial and Asymptotic Methods for the Analysis of Algorithms (AofA 2022), see \cite{Nessmann}.\end{small}

\section{Introduction and Motivation}
Suppose we are given a weighted step set $\mathcal{S}\subseteq \{-1,0,1\}^2$, and we want to count the (weighted) number $q(0,x;n)$ of excursions in the quarter plane of length $n$ from the origin to some point $x=(i,j)$. For the Simple Walk, for instance, we have $\mathcal{S}=\left\{\uparrow,\rightarrow,\downarrow,\leftarrow\right\}$, where each step has weight $\frac{1}{4}$. In this case, the number $q(0,x;n)$ can be computed explicitly (see e.g. \cite{MBM2}) via \begin{align}\label{eq:SW_exact}
    q(0,x;n)=\frac{(i+1)(j+1)n!(n+2)!}{m!(m+i+1)!(m+j+1)!(m+i+j+2)!},
\end{align}
where $m=\frac{n-i-j}{2}$. It is now fairly natural to ask about asymptotics of this expression, or more generally about asymptotics of the number $q(0,x;n)$ for an arbitrary step set $\mathcal{S}$. In particular, we consider (as proposed in \cite{Poly}) asymptotic expansions of the form \begin{align}\label{eq:asymptotic}
q(0,x;n)\sim \gamma^n\sum_{p\geq 1}\frac{v_p(x)}{n^{\alpha_p}}.
\end{align}
In case of the Simple Walk, (\ref{eq:SW_exact}) allows us to directly compute \begin{align}
    \label{eq:as_SW_1}v_1=&(i+1)(j+1),\\
    v_2=&(i+1)(j+1)(15 + 4 i + 2 i^2 + 4 j + 2 j^2),\\
    \label{eq:as_SW_2}
    \begin{split}v_3=&(i + 1) (j + 1) (317 + 16 i^3 + 4 i^4 + 168 j + 100 j^2 + 16 j^3 \\
    &   +4j^4+ 8 i (21 + 4 j + 2 j^2) + 4 i^2 (25 + 4 j + 2 j^2)).
   \end{split}
\end{align} 
It should be explicitly noted at this point that expansions of the form (\ref{eq:asymptotic}) are not proven to exist for this type of problem. While for the Simple Walk and a few other examples (e.g. the Diagonal Walk, Tandem Walk, see \cite{MBM}) this can be shown using an explicit representation similar as (\ref{eq:SW_exact}), in general it is not so clear (although one-term expansions of this form have been proven for many cases in \cite{Denisov}, and more recently, using multivariate analytic techniques, in \cite[Thm.~1]{Universality},\cite[6.1]{ACSV}).\\
It is now fairly natural to ask about the properties of the $v_p$; whether they necessarily have a particular structure, if there is a clear relation to our chosen step set, and how to compute them. And indeed, at least the first two questions can be answered fairly easily by utilizing a recursive relation between the $q(0,x;n+1)$ and $q(0,x;n)$, and showing that each function $v_p$ must be what is called a discrete polyharmonic function of order $p$.\\
In the continuous case, a function $f$ is called polyharmonic of degree $p$ if it is a solution of \begin{align}\label{eq:defPHF}
    \triangle^p f=0,
\end{align}
where $\triangle$ is a Laplacian operator $\triangle = \frac{1}{2}\left(\sigma_{11}\frac{\partial^2}{\partial x^2}+2\sigma_{12}\frac{\partial ^2}{\partial x\partial y}+\sigma_{22}\frac{\partial^2}{\partial y^2}\right)$. These kinds of functions have already been studied in the late 19th century, notably by E.~Almansi, who proved in \cite{Almansi} that in a star-shaped domain containing the origin, any polyharmonic function of degree $n$ can be written as \begin{align}
    f(x)=\sum_{k=0}^n |x|^{2k}h_k(x),
\end{align} 
where the $h_k$ are harmonic (polyharmonic of degree $1$). In particular harmonic and biharmonic functions have by now seen plenty of applications in physics, see e.g. \cite{Elasticity}.\\
The discrete setting on the other hand has gained interest comparably recently. A (discrete) function defined on a graph is called polyharmonic if it satisfies (\ref{eq:defPHF}) as well, but with a discretised version of the Laplacian. For this discretisation, given transition probabilities $p_{x,y}$ from any point $x$ to any point $y$, one lets \begin{align}
    \triangle f(x)=\sum_{y}p_{x,y}f(y)-f(x).
\end{align}
There have been some results on polyharmonic functions on trees recently \cite{Trees2,Trees1}, and polyharmonic functions on subdomains of $\mathbb{Z}^d$ have become an object of interest linked in particular to the study of discrete random walks.  In our case, this subdomain will be the quarter plane and our walk homogeneous, i.e. the transition probabilities $p_s:=p_{x,x+s}$, where the steps $s$ are given by the set $\mathcal{S}$ of allowed steps, will be independent of $x$. The discrete Laplacian thus reads \begin{align}
    \triangle f(i,j)=\sum_{(u,v)\in\mathcal{S}}p_{u,v}f(i+u,j+v)-f(i,j).
\end{align}
It is not at all obvious, however, how polyharmonic functions in general can be found. In \cite{Conformal}, a way to construct harmonic functions for zero-drift models with small steps via a boundary value problem is given. This is utilized in \cite{Hung} to give a complete discription of harmonic functions for symmetric step sets with small negative steps; the methods used therein can be applied to the case considered here with only minor adjustments. Very recently in \cite{Poly}, the authors propose a way to extend this method from harmonic to polyharmonic functions, provided one can compute a so-called `decoupling function', which was first introduced by W.~T.~Tutte in \cite{Tutte2}, and is discussed further in \cite{Tutte}. There, this concept is utilized to give remarkably succinct proofs of the algebraicity (or D-algebraicity) of the counting function of some models in the quarter plane.\\
Generally, instead of working directly with a polyharmonic function $h(u,v)$, one prefers to consider its generating function $H(x,y):=\sum_{i,j}x^{i+1}y^{j+1}h(i,j)$. The main reason to do so is the functional equation \begin{small}
\begin{align}
    \label{eq:FE} K(x,y)H(x,y)&=K(x,0)H(x,0)+K(0,y)H(0,y)-K(0,0)H(0,0)-xy\left[\triangle H\right](x,y),
\end{align}
\end{small}\noindent
which can be shown by straightforward computation to be satisfied by this generating function (note that we have $\triangle H=0$ for harmonic $H$). Here, $K(x,y)$, which will be defined in Section \ref{sec:Preliminaries}, is the same kernel that usually appears in the study of random walks, and even the resulting functional equations for counting walks or the stationary distribution look strikingly similar, see e.g. \cite{Book,MBM}.\par
This article aims to generalize and complete the notions introduced in \cite{Poly}, and thus give a description of all discrete polyharmonic functions for walks with small steps, zero drift and finite group. The main tool to do so will be the explicit computation of decoupling functions. The structure will be roughly as follows:
\begin{itemize}
\item In Section~\ref{sec:Discrete}, an algorithm to construct discrete polyharmonic functions is presented (Thm.~\ref{thm:buildPHF}). In addition, it is shown that all possible discrete polyharmonic functions can be constructed in this manner (Thm.~\ref{thm:AllPHF1}).
\item In Section \ref{sec:Continuous}, the same construction is done for the Laplace transforms of continuous polyharmonic functions (Thm.~\ref{thm:PHF_C}), again using a functional equation approach presented for biharmonic functions in \cite{Poly}.
\item In Section~\ref{sec:Convergence} the relation between the discrete and continuous cases is briefly discussed;  it is shown that discrete polyharmonic functions converge towards continuous ones in the sense of generating functions/Laplace transforms (Thm.~\ref{thm:Convergence}).
\item Lastly, Section~\ref{sec:Outlook} gives a brief overview of some open questions and ongoing research.
\end{itemize}

\section{The Discrete Case}\label{sec:Discrete}
\subsection{Preliminaries}\label{sec:Preliminaries}
The following only serves as a very brief overview; for a more thorough introduction see e.g. \cite{Conformal, Book}. Consider a homogeneous random walk in $\mathbb{Z}\times\mathbb{Z}$ with a step set $\mathcal{S}$ and transition probabilities $p_{i,j}$. From now on, we will make the following assumptions:
\begin{enumerate}[i]
    \item The walk consists of small steps only, i.e. $\mathcal{S}\subseteq\{-1,0,1\}^2$.
    \item The walk is non-degenerate, that is, the list $p_{1,1},p_{1,0},p_{1,-1},p_{0,-1},p_{-1,-1},p_{-1,0},p_{-1,1},p_{0,1}$ does not contain three consecutive $0$s. 
    \item The walk has zero drift, meaning that $\sum_{(i,j)\in\mathcal{S}}ip_{i,j}=\sum_{(i,j)\in\mathcal{S}}jp_{i,j}=0.$
    \item Any polyharmonic function considered is supposed to have Dirichlet boundary conditions, i.e. it is $0$ outside of the quarter plane. This is due to the probabilistic interpretation of (\ref{eq:asymptotic}); there can be no paths which start outside but always are inside the quarter plane.
\end{enumerate}
A standard object appearing in a variety of functional equations around random walks (besides those below for example when one wants to compute a stationary distribution, or for counting walks, see e.g. \cite{Book}) is the \textit{kernel} of the walk, which is given by \begin{align}
    K(x,y)=xy\left(\sum_{(i,j)\in\mathcal{S}}p_{i,j}x^{-i}y^{-j}-1\right).
\end{align}
In \cite{Book}, this kernel is examined quite thoroughly, and we will in the following state a few of their results.\\
As we consider non-degenerate walks with small steps, our kernel will necessarily be quadratic in both $x$ and $y$. Letting \begin{align}
\label{eq:kernelnot1}
    K(x,y)=a(x)y^2+b(x)y+c(x)= \tilde{a}(y)x^2+\tilde{b}(y)x+\tilde{c}(y),
\end{align}
we can use the quadratic formula to find solutions of $K(\cdot,y)=0$, which are given by \begin{align}
    X_\pm (y)=\frac{-\tilde{b}(y)\pm\sqrt{\tilde{b}(y)^2-4\tilde{a}(y)\tilde{c}(y)}}{2\tilde{a}(y)}.
\end{align}
One can define $Y_\pm$ in the same fashion by swapping $x,y$. Letting $D:=\tilde{b}(y)^2-4\tilde{a}(y)\tilde{c}(y)$, then one can show \cite[2.5]{Conformal}, \cite[2.3.2]{Book} that $D(y)=0$ has $3$ solutions: the double root $y=1$, a solution $y_1\in [-1,1)$, and a solution $y_4\in (1,\infty)\cup (-\infty, -1]$. Consequently, one can see that for $y\in[y_1,1]$, we have $X_+(y)=\overline{X_-(y)}$. This is in particular used in the computation of harmonic functions, as in \cite{Conformal} or \cite{Hung}. The idea is to define the domain $\mathcal{G}$ as the area bounded by the curve $X_\pm\left([y_1,1]\right)$, and notice that the functional equation~(\ref{eq:FE}) leads to the boundary value problem \begin{align}
    K(x,0)H(x,0)-K(\overline{x},0)H(\overline{x},0)=0
\end{align}
on $\partial\mathcal{G}\setminus\{1\}$, while $K(x,0)H(x,0)$ is analytic in the interior of $\mathcal{G}$ and continuous on $\overline{\mathcal{G}}\setminus\{1\}$ (cf \cite{Poly,Conformal}). One can then construct a mapping $\omega: \mathbb{C}\to \bar{\mathbb{C}}$ which is a fundamental solution of the above BVP, in the sense that any other solution can be written as some entire function applied to $\omega$ \cite{Conformal}. This $\omega$ then has the properties \begin{align}\label{eq:wInv}
    \omega(0)&=0,\quad
    \omega(X_+(y))=\omega(X_-(y))\quad\forall y\in[y_1,1],\quad 
   \frac{\partial \omega}{\partial x}(x)\neq 0\quad\forall x\in\mathcal{G}^{\circ}.
\end{align}
In particular, $\omega$ is a conformal mapping of the domain $\mathcal{G}$. Furthermore, it has a pole-like singularity of order $\pi/\theta$ at $x=1$, where $\theta$ is the inner angle at which $\partial\mathcal{G}$ intersects the $x$-axis. Alternatively, $\theta$ can be computed via \begin{align}\label{eq:defTheta}
    \theta=\arccos\left(-\frac{\sum ij p_{i,j}}{\sqrt{\sum i^2 p_{i,j}}\sqrt{\sum j^2 p_{i,j}}}\right),
\end{align} see e.g. \cite[2.15]{Conformal}. This angle is also closely related to the group of a walk, which will be introduced below in Section~\ref{sec:Decoupling}. Additionally, we can see that $\omega\circ X_+$ is a conformal mapping of a region $\mathcal{G}'$ obtained by swapping $x,y$ (by \cite[Cor.~5.3.5]{Book}), and it has the same behaviour around $1$ as $\omega$. Finally, $\omega$ turns out to be an invariant in the sense of \cite[Def.~4.3]{Tutte}.
\subsection{Discrete Polyharmonic Functions}
We will start with a few elementary properties. Denote in the following by $\mathcal{H}_n$ the space of real-valued discrete $n$-polyharmonic functions. Clearly, $\mathcal{H}_n$ is a $\mathbb{R}$-vector space. Now, given any $\widehat{H}_n\in\mathcal{H}_n$, we can identify it with the sequence $\left(\widehat{H}_n,\widehat{H}_{n-1},\dots, \widehat{H}_1\right)$, where $\triangle\widehat{H}_{k+1}=\widehat{H}_k$, and $\triangle\widehat{H}_1=0$. It is clear that any such sequence is uniquely defined by $\widehat{H}_n$. Now suppose that we have $\widehat{H}_n,\widehat{H}_n'\in \mathcal{H}_n$, such that, with their sequence representation as above, $\widehat{H}_1=\widehat{H}_1'$. In this case, we have \begin{align}
    \triangle^{n-1}\left[\widehat{H}_n-\widehat{H}_n'\right]=\widehat{H}_1-\widehat{H}_1'=0,
\end{align}
thus $\widehat{H}_n-\widehat{H}_n'\in\mathcal{H}_{n-1}$. Therefore, provided that for each $\widehat{H}_n\in\mathcal{H}_n$ we can find a corresponding $\widehat{H}_{n+1}\in\mathcal{H}_{n+1}$, which will be shown below in Thms.~\ref{thm:buildPHF} and \ref{thm:Convergence}, by simple induction one can prove the following Lemma:
\begin{lemma}
Let $\mathcal{H}_n$ be the space of real-valued, discrete $n$-polyharmonic functions in the quarter plane, subject to $\widehat{H}_n(0,\cdot)=\widehat{H}_n(\cdot,0)\equiv 0$. Then we have an isomorphy of vector spaces \begin{align}
    \mathcal{H}_n\cong \left(\mathcal{H}_1\right)^n.
\end{align}
\end{lemma}
In particular, if we are given any $\widehat{H}_n\in\mathcal{H}_n$, and we want to find all corresponding $\widehat{H}_{n+1}\in\mathcal{H}_{n+1}$, then this means that instead it suffices to find a single $\widehat{H}_{n+1}$ with this property as well as all harmonic functions, because any other such $\widehat{H}_{n+1}'$ can be written as $\widehat{H}_{n+1}+\widehat{G}_1$, for some $\widehat{G}_1\in\mathcal{H}_1$. This property enables us to completely classify discrete harmonic functions.\\
We already know (see e.g. \cite{Conformal},\cite{Poly}, using the idea of the BVP outlined above), that for any polynomial $P(x)\in\mathbb{R}[x]$, we can construct (the GF of) a harmonic function via 
\begin{align}\label{eq:buildHarmonic}
    H(x,y)=\frac{P\left(\omega(x)\right)-P\left(\omega\left(X_+(y)\right)\right)}{K(x,y)},
\end{align}
where $\omega$ is the conformal mapping introduced in Section~\ref{sec:Preliminaries}. We will now show what is, in a sense, the opposite direction of the above statement. The following theorem (as well as its proof) is an analogue to \cite[Thm.~2]{Hung}, where a similar result is shown for the case of symmetric walks with small negative steps. 
\begin{restatable}{theorem}{thmAllHF}\label{thm:AllHF}
For any discrete harmonic function with generating function $H(x,y)$, there is a unique formal power series $P(x)$ such that (\ref{eq:buildHarmonic}) holds. In particular, we have an isomorphism \begin{align}
    \mathcal{H}_1\cong \mathbb{R}[[x]].
\end{align}
\end{restatable}

\begin{proof}  One constructs explicitly a basis  $\left(H_1^k\right)_{k\in\mathbb{N}}$ via \begin{align}\label{eq:defH1k}
H_1^k(x,y):=\frac{P_k\left(\omega(x)\right)-P_k\left(\omega\left(X_+\right)\right)}{K(x,y)},
\end{align}
where $P_{2k}(x)=x^k(x-d_0)^k, P_{2k+1}=x^{k+1}(x-d_0)^k$ for $d_0=\omega\left(X_+(0)\right).$ See App. \ref{app:thmAllHF}. \end{proof}
If we compare the functional equation (\ref{eq:FE}) for harmonic and polyharmonic functions, then the only difference lies in the additional term of $xyH_n(x,y)$ on the right-hand side not vanishing for the latter. In terms of a BVP, this means that
we now want to solve \begin{small}\begin{align}\label{eq:DecouplingBVP}
    K\left(X_+,0\right)H_n\left(X_+,0\right)-K\left(X_-,0\right)H_n\left(X_-,0\right)=X_+yH_{n-1}(X_+,y)-X_-yH_{n-1}(X_-,y).
\end{align}
\end{small}\noindent
In an ideal world, the right-hand side of the latter equation would be $0$ as in the harmonic case, and this is indeed what happens for the Simple Walk. In this case, we can proceed as before, and obtain an explicit formula for polyharmonic functions.

\subsection{Example: the Simple Walk}\label{sec:SimpleWalk}
The Simple Walk has the step set $\mathcal{S}=\{\uparrow, \rightarrow,\downarrow,\leftarrow\}$, each with probability $\frac{1}{4}$. We have \begin{align}
    K(x,y)=\frac{xy}{4}\left(x+y+x^{-1}+y^{-1}\right)-xy,\quad
    \omega(x)=\frac{-2x}{(1-x)^2},\quad 
    \omega\left(X_+\right)=-\omega(y).
\end{align}
As it turns out that the right-hand side of (\ref{eq:DecouplingBVP}) keeps vanishing, one can iteratively construct polyharmonic functions via $H_{n+1}(x,y):=\frac{xyH_n(x,y)-X_+yH_n(X_+,y)}{K(x,y)}$. This allows us to find an explicit expression for all resulting polyharmonic functions. It appears that this property is directly tied to the fact that $\pi/\theta=2$, where $\theta$ is given by (\ref{eq:defTheta}). 

\begin{restatable}{theorem}{thmSimpleWalk}\label{thm:SimpleWalk}\text{}
\begin{enumerate}
    \item 
The functions defined by \begin{align}
    H_m^k(x,y)=2^{m-1}\omega(X_0)^{m-1}\frac{\omega(x)-\omega\left(X_+\right)}{K(x,y)}\left[\sum_{j=0}^{k-1}s_{m}(j)\omega\left(X_+\right)^j\omega(x)^{k-j-1}\right],
\end{align}
where $s_l: \mathbb{N}\to\mathbb{N}$ is defined inductively via $s_1(j)=1, s_{l+1}(j)=\sum_{i=1}^{j+1}s_l(j)$, are polyharmonic functions with $\triangle H_{m+1}^k=H_m^k$. 
\item Given any formal power series $P(x)=\sum a_nx^n$ and any $m$, the limit \begin{align}
    H_{m}^P(x,y):=\lim_{n\to\infty}\sum_{j=0}^na_jH_m^j
\end{align}
exists, and we again have $\triangle H_{m+1}^P=H_m^P$. In particular, any discrete $m$-polyharmonic function can be written as $H_m^P$ for some $P$.
\end{enumerate}
\end{restatable}
\begin{proof}
See App.~\ref{app:thmSW}.\end{proof}
Using Thm.~\ref{thm:SimpleWalk}, we can e.g. directly compute 
$H_1^1=-\frac{8}{(1-x)^2(1-y)^2}$, $H_2^1=-\frac{32y}{(x-1)^2(y-1)^4}$, $H_3^1=-\frac{128y^2}{(x-1)^2(y-1)^6}$.\\
Proceeding to compute the generating functions $V_{1,2,3}$ of $v_{1,2,3}$ as given in (\ref{eq:as_SW_1})--(\ref{eq:as_SW_2}), we obtain \begin{small}\begin{align*} 
V_1=64H_1^1,\quad V_2=\frac{3}{8}H_2^1-\frac{3}{8}H_1^2+60H_1^1,\quad 
V_3=-24H_3^1+24H_2^2+72H_2^1-30H_1^3-72H_1^2+5072H_1^1.
\end{align*}
\end{small}\noindent
It is somewhat striking here that the only $p$-polyharmonic part contained in $v_p$ is $H_p^1$, which is in some manner the simplest possible. At this stage there is neither a proof that this is always true nor a counter-example.

\subsection{Decoupling}\label{sec:Decoupling}
While the computation for the Simple Walk turned out to be fairly simple, this was mainly due to the right-hand side of (\ref{eq:DecouplingBVP}) consistently vanishing. This does not happen in general. For the Tandem Walk, for instance, we arrive at 
\begin{align}
    K(X_+,0)H_1^1(X_+,0)-K(X_-,0)H_1^1(X_+,0)=\frac{y^3\sqrt{1-4y}}{(y-1)^5}.
\end{align}
The direct approach using the BVP like in the harmonic case does not generally yield an explicit solution as easily as before. Instead, we choose a more combinatorial approach and utilize what is called a decoupling function in \cite[Def. 4.7]{Tutte}. 
\begin{definition} Let $M(x,y)$ be an expression in $x,y$. If we can find $F(x), G(y)$ such that \begin{align}\label{eq:def:Decoupling}
    F(x)+G(y)\equiv M(x,y)\quad\operatorname{mod} K(x,y),
\end{align}
then we say that $F$ is a \textbf{decoupling function} of $M$.
\end{definition}
These decoupling functions are closely related to the concept of invariants as in \cite[Def.~4.3]{Tutte}. An example of a decoupling function will for instance be given in Section~\ref{sec:exampleTandem}. Let in the following $H'$ be polyharmonic, and $H$ be such that $\triangle H=H'$. By substitution into (\ref{eq:DecouplingBVP}), we directly find that for $F(x)$ a decoupling function of $xyH'(x,y)$ we have \begin{align}\label{eq:def:Decoupling2}
    K(X_+,0)H(X_+,0)-F(X_+)-\left[K(X_-,0)H(X_-,0)-F(X_-)\right]=0.
\end{align}
In other words, if one knows how to compute a decoupling function of $xyH'(x,y)$, then one can again let $K(x,0)H(x,0)-F(x)=P(\omega)$; by the same arguments as for the BVP above one will then eventually arrive at a solution for $H(x,y)$. In \cite[App.~C]{Poly}, a decoupling function is guessed in order to compute a biharmonic function for the Tandem Walk. It turns out, however, that such a decoupling function can be explicitly computed for any model as long as a certain group of the corresponding step set is finite. This group is generated by the mappings \begin{align}
    \Phi: \quad \begin{cases} x&\mapsto x^{-1}\frac{\tilde{c}(y)}{\tilde{a}(y)},\\
    y&\mapsto y,\end{cases}\quad\quad\quad
    \Psi: \quad \begin{cases} x&\mapsto x,\\ y&\mapsto y^{-1}\frac{c(x)}{a(x)}.\end{cases}
\end{align}
One can easily see that $\Phi,\Psi$ are involutions, and depending on the order of $\Theta:=\Phi\circ\Psi$, the group can be either finite or infinite. This group has been of interest in the study of random walks for some time now, see e.g. \cite{MBM,Singer}.
\begin{theorem}[{see \cite[Thm.~4.11]{Tutte}}]
Suppose our step set has a finite group of order $2n$, and $M(x,y)$ is such that \begin{align}\label{eq:SignedOrbit}
    \sum_{\gamma\in\mathcal{G}}\operatorname{sgn}(\gamma)\gamma\left(M(x,y)\right)=0.
\end{align}
Then a decoupling function of $M(x,y)$ is given by \begin{align}\label{eq:DecouplingFormula}
    F(x)=-\frac{1}{n}\sum_{i=1}^{n-1}\theta^i\left[M(x,Y_+)+M(x,Y_-)\right].
\end{align}
\end{theorem}
In the following, we will show that $xyH_n(x,y)$ will turn out to have an orbit sum of $0$ for any polyharmonic $H_n$. This is in particular independent of whether or not the given model has a vanishing orbit sum as in \cite{Tutte}. 

\begin{corollary}\label{prop:oSum}
Suppose the group of the step set is finite and has a series representation around $(0,0)$. Then any rational function $M(x,y)$ of the form \begin{align}
    M(x,y)=xy\frac{u(x)+v(y)}{K(x,y)}
\end{align}
allows for a decoupling function via (\ref{eq:DecouplingFormula}).
\end{corollary}
\begin{proof}
This follows directly from the fact that the denominator $\frac{1}{xy}K(x,y)$ is invariant under $\mathcal{G}$. Alternating orbit summation over the numerator leads to a telescopic sum.
\end{proof}
If a model has a finite group, then it can be shown that $\pi/\theta\in\mathbb{Q}$ (cf \cite[7.1]{Book}). While the following is only stated for $\pi/\theta\in\mathbb{Z}$, it is possible to extend the statements for an arbitrary finite group. This relaxation of conditions, however, adds a lot more technicalities as the resulting functions are not rational anymore.

\begin{restatable}{theorem}{thmbuildPHF}\label{thm:buildPHF}
Suppose our step set has finite group and $\pi/\theta\in\mathbb{Z}$. Let $H_1^k(x,y)$ be defined by (\ref{eq:defH1k}). We can then define inductively \begin{align}\label{eq:buildPHF}
    H_{n}^k(x,y)=\frac{xyH_{n-1}^k(x,y)-F_{n-1}^k(x)-\left[X_+yH_{n-1}^k(X_+,y)-F_{n-1}^k(X_+,y)\right]}{K(x,y)},
\end{align}
where $F_n^k(x)$ is the decoupling function of $xyH_n^k(x,y)$ defined by (\ref{eq:DecouplingFormula}). Then, $H_n^k(x,y)$ is a rational function in $\mathcal{H}_n$ for all $n,k$, which satisfies $\triangle H_{n+1}^k=H_n^k$ as well as (\ref{eq:SignedOrbit}). For each $n,k$ we can write \begin{align}\label{eq:formPHF}
    H_n^k(x,y)=\frac{p_{n,k}(x,y)}{(1-x)^{\alpha}(1-y)^{\alpha}},
\end{align}
where $p_{n,k}$ is a polynomial and $\alpha=k\cdot\pi/\theta+2(n-1)$. 
\end{restatable}

\begin{proof} See App.~\ref{app:thmbuildPHF}.\end{proof}
The downside of this construction is that we do not know for sure that for any $k$, the sum $\sum_{n=1}^\infty H_n^k$ converges. This property would be very useful in the proof that we can indeed find all polyharmonic functions. However, utilizing the functional equation (\ref{eq:FE}) and proceeding similarly as in the proof of Thm.~\ref{thm:AllHF}, one can for each induction step find a harmonic function $H_{1,n}^k$ such that $\widehat{H_n^k}:=H_n^k+H_{1,n}^k$ has order at least $\left[\frac{k}{2}\right]$ at $0$. The thusly defined $\widehat{H_n^k}$ therefore satisfy the conditions of the following theorem.

\begin{theorem}\label{thm:AllPHF1}
Let $\left(H_n^k\right)_{n,k\in\mathbb{N}}$ be a family of discrete polyharmonic functions, such that \begin{enumerate}
    \item $H_1^k=H_1^k$ is given by (\ref{eq:defH1k}),
    \item $\triangle H_{n+1}^k=H_n^k$,
    \item For any $n$ and any sequence $(a_n)$, the sum $\sum_{n=1}^\infty a_n H_n^k$
    converges.
\end{enumerate}Then, given any $H_n\in\mathcal{H}_n$, we can find $a_{i,j}, 1\leq i\leq n, j\in\mathbb{N}$, such that \begin{align}
    H_n=\sum_{i=1}^n\sum_{j=1}^\infty a_{i,j}H_i^j.
\end{align}

\end{theorem}
\begin{proof}
By induction. For $n=1$, the statement is nothing but Thm. \ref{thm:AllHF}. Now assume the theorem holds for $n$, and let $H_{n+1}\in\mathcal{H}_{n+1}$. By definition, we must then have $H_n:=\triangle H_{n+1}\in \mathcal{H}_n$, so we can write $H_n=\sum_{i=1}^n\sum_{j=1}^\infty a_{i,j}H_i^j$. By construction, $\overline{H_{n+1}}:=\sum_{i=1}^n\sum_{j=1}^\infty a_{i,j}H_{i+1}^j$ is then a $n+1$-polyharmonic function with $\triangle \overline{H_{n+1}}=H_n$. Thus, $H_{n+1}-\overline{H_{n+1}}\in\mathcal{H}_1$, and an application of Thm.~\ref{thm:AllHF} and relabeling of the coefficients immediately yields the statement.
\end{proof}
As an immediate consequence, we can state that any discrete polyharmonic function can be expressed as a countable sum of our $H_n^k$ as defined by (\ref{eq:buildPHF}). 

\subsection{Example: the Tandem Walk}\label{sec:exampleTandem}
To illustrate the results from Section~\ref{sec:Decoupling}, consider the Tandem Walk, which has the step set $\mathcal{S}=\{\rightarrow, \downarrow, \nwarrow\}$, with weights $\frac{1}{3}$ each. We have
\begin{align}
    K(x,y)=\frac{xy}{3}\left(x^{-1}+y+xy^{-1}\right)-xy,\quad
    \omega(x)=\frac{27x^2}{4(x-1)^3},\quad
    \omega\left(X_+\right)=\frac{-27y}{4(y-1)^3}.
\end{align}
We directly obtain $H_1^1=\frac{\omega(x)-\omega(X_+)}{K(x,y)}=\frac{81(xy-1)}{4(x-1)^3(y-1)^3}$, leading to the harmonic function $h(i,j)=(i+1)(j+1)(i+j+2)$. Using (\ref{eq:DecouplingFormula}), we obtain the decoupling function $F_1(x)=-\frac{81x^3}{4(1-x)^5}$. Note that this decoupling function is not the same one as is given in \cite[App. C]{Poly}, where instead (after scaling) $F_1'=\frac{-81x^3}{4(1-x)^6}$ is given. This goes to show that the choice of a decoupling function is, due to the invariance property in (\ref{eq:wInv}), unique only up to functions of $\omega$; in this particular case we have (up to a multiplicative constant) $F_1'(x)-F_1(x)=\omega(x)^2$.\\
Using $F_1$ in (\ref{eq:buildPHF}) directly gives us the biharmonic function $H_2^1=-\frac{243(xy-1)(x+y+xy(x+y-4))}{(x-1)^5(y-1)^5}$.\\
Once again using (\ref{eq:DecouplingFormula}) gives us the next decoupling function $F_2(x)=\frac{81x^2(x+2)}{4(x-1)^7}$, which we can then use to compute $H_3^1=\frac{p(x,y)}{(x-1)^7(y-1)^7}$, for $p(x,y)$ a somewhat unwieldy polynomial of degree $9$.

\section{The Continuous Case}

We now consider solutions of (\ref{eq:defPHF}) with the usual Laplacian. First, it needs to be made clear in which way the latter corresponds to a given step set.  So instead of a discrete random walk on $\mathbb{Z}_{\geq 0}\times\mathbb{Z}_{\geq 0}$, we can also consider a Brownian motion on $\mathbb{R}^+\times\mathbb{R}^+$. Any such Brownian motion is defined by its covariance matrix $\Sigma = \begin{pmatrix} \sigma_{11} & \sigma_{12}\\ \sigma_{12} & \sigma_{22}\end{pmatrix}$, and its infinitesimal generator is the Laplacian \begin{align}\label{eq:defContLap}
    \triangle=\frac{1}{2}\left(\sigma_{11}\frac{\partial^2}{\partial x^2}+2\sigma_{12}\frac{\partial ^2}{\partial x\partial y}+\sigma_{22}\frac{\partial^2}{\partial y^2}\right),
\end{align}
consequently any polyharmonic function with respect to this Brownian motion satisfies (\ref{eq:defPHF}). It can be shown that this Brownian motion is the scaling limit of any non-degenerate discrete random walk with small steps and zero drift such that $\mathbb{E}X^2=\sigma_{11},
    \mathbb{E}XY=\sigma_{12},
    \mathbb{E}Y^2=\sigma_{22}$ \cite{Limic}.
While it is possible to compute solutions of (\ref{eq:defPHF}) explicitly, for instance using polar coordinates as in \cite{Poly,Griffiths}, one can also follow an approach of \cite[App.~A]{Conformal}. There, instead of a generating function, the authors consider the Laplace transform and obtain the following functional equation:
\begin{align}
    \gamma(x,y)L(h)(x,y)=\frac{1}{2}\left(\sigma_{11}L_1(h)(y)+\sigma_{22}L_2(h)(x)\right)+L\left(\triangle(h)\right)(x,y),
\end{align}
where \begin{small}\begin{align}
    \gamma(x,y)&=\frac{1}{2}\left(\sigma_{11}x^2+2\sigma_{12}xy+\sigma_{22}y^2\right),\quad L(f)(x,y)=\int_0^\infty\int_0^\infty e^{-ux-vy}f(u,v)\mathrm{d}u\mathrm{d}v,\\
    L_1(f)(y)&=\int_0^\infty \frac{\partial f}{\partial x}(0,v)e^{-vy}\mathrm{d}v,\quad\qquad\qquad
    L_2(f)(x)=\int_0^\infty\frac{\partial f}{\partial y}(u,0)e^{-ux}\mathrm{d}u,
\end{align}
\end{small}\noindent
see also \cite[2.2]{Poly}. For any polyharmonic function $h_n$, this yields \begin{align}\label{eq:cont_PHF}
    \gamma(x,y)L(h_n)(x,y)=\frac{1}{2}\left(\sigma_{11}L_1(y)+\sigma_{22}L_2(x)\right)+L(h_{n-1})(x,y),
\end{align}
where $h_{n-1}=\mathcal{L}(h_n)$ (for $n=1$ we let $h_0=0$). This functional equation, which is very similar to (\ref{eq:FE}), can now be utilized in order to compute continuous polyharmonic functions. For the harmonic and biharmonic case this has already been done in \cite[2.2]{Poly}, where it is also mentioned that a similar method should work to compute higher order polyharmonic functions. Such a method shall be presented in the following.

\subsection{Continuous Polyharmonic Functions}\label{sec:Continuous}
The idea used in computing harmonic functions in \cite{Conformal, Poly} is very much the same as in the discrete setting. We can rewrite and obtain as in Section~\ref{sec:Preliminaries}\begin{align}
    \gamma(x,y)=\frac{1}{2}\left(\sigma_{11}x^2+2\sigma_{12}xy+\sigma_{22}y^2\right),\quad
    x_\pm=c_\pm y,\quad
    \widehat{\omega}(x)=\frac{1}{x^{\pi/\theta}},
\end{align}
where $c_\pm=c e^{\pm i\theta}$, and eventually arrive at the following result. \begin{proposition}[{\cite[Th.~2.4]{Poly}}]
For any polynomial $P(x)$, the function \begin{align}
    L(h^P)(x,y):=\frac{P\left(\omega(x)\right)-P\left(\omega(c_+ y)\right)}{\gamma(x,y)}
\end{align}
is the Laplace transform of a harmonic function. 
\end{proposition}
The equivalent for the BVP for continuous polyharmonic functions now reads \begin{align}
    \sigma_{22}L_2(h_n)(c_+y)-\sigma_{22}L_2(h_n)(c_-y)=L(h_{n-1}(c_+y,y))-L(h_{n-1}(c_-y,y)),
\end{align}
where as in Section~\ref{sec:Decoupling} the difference to the harmonic case is the right-hand side not necessarily being equal to $0$. Our goal will now be to construct a decoupling function $f_{n-1}(x)$, such that \begin{align}\label{eq:cont_decoupling}
f_{n-1}(c_+y)-f_{n-1}(c_-y)=L(h_{n-1})(c_+y,y)-L(h_{n-1})(c_-y,y).
\end{align}
The computation of such a decoupling function here turns out to be a lot simpler than in the discrete case, as can be seen in the example of the scaling limit of the Tandem Walk.

\subsection{Example: the Scaling Limit of the Tandem Walk}\label{sec:exampleTandemC}
For the scaling limit of the Tandem Walk, we have \begin{align}
    \gamma(x,y)=\frac{1}{3}\left(x^2-xy+y^2\right),\quad
    c_\pm=\frac{1\pm i\sqrt{3}}{2},\quad
    \widehat{\omega}(x)=\frac{1}{x^3}.
\end{align}
Selecting $L(h_1)(x,y)=\frac{\widehat{\omega}(x)-\widehat{\omega}(c_+y)}{\gamma(x,y)}=\frac{3(x+y)}{x^3y^3}$,
(\ref{eq:cont_decoupling}) takes the form $f_1(c_+y)-f_1(c_-y)=\frac{3i\sqrt{3}}{y^5}$.
Since the right-hand side is homogeneous (which, in fact, is generally true, seeing as both $\widehat{\omega}$ and $\gamma$ are homogeneous), the ansatz $f_1(x)=\frac{\alpha}{x^5}$ is very reasonable. By a quick computation, one obtains that $\alpha$ must be $-1$. Everything works out, we obtain $f_1(x)=\frac{-3}{x^5}$ and a biharmonic function
\begin{small}\begin{align}
    L(h_2)(x,y)=\frac{L(h_1)(x,y)-f_1(x)-\left[L(h_1)(c_+y,y)-f_1(c_+y)\right]}{\gamma(x,y)}=\frac{9(x+y)(x^2+y^2)}{x^5y^5}.
\end{align}
\end{small}

\subsection{Decoupling}\label{sec:Decoupling_C}
In the following, we let (analogously to the discrete case) \begin{align}\label{eq:buildH_C}
    L(h_1^k)(x,y):=\frac{\widehat{\omega}(x)^k-\widehat{\omega}(c_+y)^k}{\gamma(x,y)}.
\end{align}
To compute polyharmonic functions, all we need are decoupling functions, i.e. that we can always proceed as in the example above. Given a (homogeneous) polyharmonic function $L(h)$, this can be seen to directly depend on its degree. If $\deg L(h)$ is divisible by $\pi/\theta$, then our ansatz would not work, for then $c_+^{-\deg L(h)}=c_-^{-\deg L(h)}$. However, this turns out not to matter: were one to continue the above example of the Tandem Walk until the first case where this could be an issue, that is, computing $f_3$, then one would see that we already have $L(h_3)(c_+y,y)=L(h_3)(c_-y,y)$, meaning we can select $f_3=0$. This is not a coincidence and will always happen in those cases. Since this is essential in order to continue the procedure but our proof relies heavily on convergence of discrete polyharmonic functions, it will be stated here and be proven in Section~\ref{sec:Convergence}.

\begin{lemma}\label{lemma:contDecoupling}
The procedure starting from a $L(h_1^k)$ described in Section~\ref{sec:exampleTandemC} always works, meaning that there is a constant $\alpha$ such that we can find a decoupling function of the form  \begin{align}f_n(x)=\frac{\alpha}{x^{\deg L(h_n)}}.
\end{align}
\end{lemma}

Utilizing the above Lemma, it is now easy to prove the continuous analogue of Thm.~\ref{thm:buildPHF}.
\begin{theorem}\label{thm:PHF_C}
Let $L(h_1^k)(x,y)$ be defined by (\ref{eq:buildH_C}). We can then define inductively \begin{align}\label{eq:buildPHF_C}
    L(h_n^k)(x,y)=\frac{L(h_{n-1})(x,y)-f_{n-1}(x)-\left[L(h_{n-1})(c_+y,y)-f_{n-1}(c_+y)\right]}{\gamma(x,y)},
\end{align}
where $f_n(x)$ is a decoupling function as in Section~\ref{sec:exampleTandemC}. Then, $L(h_n^k)(x,y)$ is the Laplace transform of an $n$-harmonic function, such that $\mathcal{L}h_n^k=h_{n-1}^k$. For each $n,k$ we can write \begin{align}
    L(h_n^k)(x,y)=\frac{q_{n,k}(x,y)}{x^\alpha y^\alpha},
\end{align}
for $\alpha=k\pi/\theta+2n$ and  $q_{n,k}(x,y)\in\mathbb{C}\left[x,y,x^{\pi/\theta},y^{\pi/\theta}\right]$. In particular, if $\pi/\theta\in \mathbb{Z}$, then $q_{n,k}(x,y)$ is a polynomial. Furthermore, $q_{n,k}$ is homogeneous of degree $k\pi/\theta-2+4n$.
\end{theorem}

\begin{proof}
For $n=1$, the statement can be checked directly. The rest follows inductively using Lemma \ref{lemma:contDecoupling}. 
\end{proof}

\section{Convergence}\label{sec:Convergence}

Having continuous and discrete polyharmonic functions associated with a step set, it is natural to assume there would be some kind of connection between them. And indeed, this turns out to be the case, in a similar fashion as has been shown for harmonic functions of symmetric walks in \cite[Th.~1]{Hung}. Comparing the Laplace transform and the generating functions of some functions $h$ and $g$, we have \begin{align}
    L(h)(x,y)=\int_0^\infty\int_0^\infty e^{-ux-vy}h(u,v)\mathrm{d}u\mathrm{d}v,\quad 
    G(x,y)=\sum_{i=0}^\infty\sum_{j=0}^\infty x^iy^j g(i,j).
\end{align}
It is therefore fairly natural to consider expressions of the form $H\left(e^{-x},e^{-y}\right)$. To transition from the discrete to the continuous setting we will also need some scaling parameter, which eventually leads to us considering limits of the form $\lim_{\mu\to 0}\mu^\alpha H\left(e^{-\mu x},e^{-\mu y}\right)$ for some constant $\alpha$. That this kind of limit is reasonable is further shown by the following relations between expressions used in the discrete and continuous cases respectively.\\
The following Lemma~\ref{lemma:limits} can be proven by straightforward computation.

\begin{lemma}\label{lemma:limits}
Let the discrete and continuous kernels $K$ and $\gamma$ belong to the same step set. We then have \begin{align}
    \label{eq:kernellimit} \lim_{\mu\to 0} \frac{K\left(e^{-\mu x},e^{-\mu y}\right)}{\mu^2}=\gamma(x,y),\quad
\lim_{\mu\to 0} X_\pm\left(e^{-\mu x}\right)=1-c_\pm \mu x +\mathcal{O}(\mu ^2).
\end{align}
\end{lemma}
Using Lemma \ref{lemma:limits}, the strategy in order to show a general convergence of polyharmonic functions is quite simple: we use the fact that the recursive definitions (\ref{eq:buildPHF}) and (\ref{eq:buildPHF_C}) have the same structure, and take the limit of each term separately. All that remains is to consider decoupling functions. However, using once again Lemma \ref{lemma:limits}, this turns out to be rather straightforward, too.

\begin{lemma}\label{lemma:limit_Decoupling}
Suppose we have a discrete and continuous polyharmonic function $H$ and $L(h)$, and a constant $\alpha$ such that \begin{align}
    \lim_{\mu\to 0}\mu^\alpha H\left(e^{-\mu x},e^{-\mu y}\right)=L(h)(x,y).
\end{align}
Then, for any decoupling function $F(x)$ of $xyH(x,y)$,\begin{align}
    f(x):=\lim_{\mu\to 0} \mu^\alpha F\left(e^{-\mu x}\right)
\end{align}
is a decoupling function of $L(h)(x,y)$. 
\end{lemma}

\begin{proof}
By taking the corresponding limit of (\ref{eq:def:Decoupling2}).
\end{proof}
We can now formulate and prove the following theorem, which shows convergence between the $H_n^k$ and the $L(h_n^k)$ defined in Sections~\ref{sec:Decoupling} and~\ref{sec:Decoupling_C} respectively. In doing so, we will also prove Lemma~\ref{lemma:contDecoupling}. Since we will be using Thm.~\ref{thm:PHF_C} to do so, which in turn utilizes the former, it is worth taking a moment to make sure that in each induction step in the proof of Thm.~\ref{thm:Convergence} for some fixed $n+1$, we use the statement of Thm.~\ref{thm:PHF_C} for $n$, and prove Lemma \ref{lemma:contDecoupling} for $n+1$. We therefore do not enter any circular reasoning.  

\begin{theorem}\label{thm:Convergence}
Let $\pi/\theta\in\mathbb{Z}$ and $H_n^k,L(h_n^k)$ be defined by (\ref{eq:buildPHF}),~(\ref{eq:buildPHF_C}) respectively. Then \begin{align}
    \lim_{\mu\to 0}\mu^{k\pi/\theta+2n}H\left(e^{-\mu x},e^{-\mu y}\right)=\alpha_{n,k}L\left(h_n^k\right)(x,y)
\end{align}
for some constants $\alpha_{n,k}\neq 0$.
\end{theorem}
\begin{proof}[Proof of Thm.~\ref{thm:Convergence} and Lemma \ref{lemma:contDecoupling}]
For $n=1$ the statement can be checked by direct computation. Now let the statement be true for some $n$. By Lemma \ref{lemma:limit_Decoupling}, we know that we can define a decoupling function of $L(h_n^k)(x,y)$ via $f_n^k(x)=\frac{1}{\alpha_{n,k}} \lim_{\mu\to 0}\mu^{k\pi/\theta+2n}F_n^k\left(e^{-\mu x}\right)$. From the proof of Thm.~\ref{thm:buildPHF}, we know that $F_n^k(x)$ has the form $P_{n,k}(x)/(1-x)^{2n+k\pi/\theta}$ for some polynomial $P_{n,k}(x)$. Therefore, the ansatz used in Ex. \ref{sec:exampleTandemC} must yield a solution, so Lemma \ref{lemma:contDecoupling} is proven. Having now completely proven Thm.~\ref{thm:PHF_C} for $n+1$, we can simply take piecewise limits of (\ref{eq:buildPHF}), (\ref{eq:buildPHF_C}), and obtain the statement.
\end{proof}

\section{Outlook/Open Questions}\label{sec:Outlook}
While Thm.~\ref{thm:buildPHF} is formulated only for the case $\pi/\theta\in\mathbb{N}$ here, the construction of polyharmonic functions works in essentially the same manner for $\pi/\theta\in\mathbb{Q}$ (albeit the process becomes a bit more technical as we need to work with algebraic instead of rational functions). This will be addressed in a future paper, together with more complete proofs of the theorems above.\begin{itemize}
    \item Seeing as for any finite group we have $\pi/\theta\in\mathbb{Q}$ \cite[7.1]{Book} (though the reverse does not hold), it follows that we can construct arbitrary polyharmonic functions for any walk with drift $0$, small steps and finite group. This directly leads to the question of what to do in the infinite group case. There are some examples where one can show by a direct ansatz that a decoupling function of reasonably nice shape cannot exist, and it can indeed be conjectured that a decoupling function can be computed if and only if the group of the walk is finite. Should this hold, then of course the next question would be how discrete polyharmonic functions could be computed in the infinite group case.
    \item Another open question regards the positivity of harmonic functions. For drift $0$ walks, we know \cite{MartinBoundary} that there is a unique (up to multiples) positive harmonic function. In all known examples this harmonic function is given by $H_1^1$ as in (\ref{eq:defH1k}), and for the symmetric case it was already conjectured in \cite[Conj.~1]{Hung} that this might be true in general. To the author's knowledge there has not yet been a general result in that direction.
    \item While the above gives a description of all discrete polyharmonic function, it is still not clear which ones appear in expansions of the form (\ref{eq:asymptotic}). If the conjecture about the positive harmonic function were to be true, then the harmonic function $v_1$ would have to be a multiple of $H_1^1$. One could then proceed to ask whether this turns out to be true in a more general setting, i.e. if the $p$-harmonic part of $v_p$ as given by (\ref{eq:asymptotic}) is always a multiple of $H_p^1$, as is the case in example \ref{sec:SimpleWalk}.
    \item In all known (at least to the author) enumeration problems of lattice paths with small steps and zero drift in the quarter plane that we can compute higher-order asymptotics of, an expansion as in (\ref{eq:asymptotic}) exists. This leads to the question of whether there are models for which there is no such expansions, or if there are conditions under which the existence of the latter can be shown in general.
\end{itemize}

\section*{Acknowledgements}
I would like to thank Kilian Raschel for introducing me to this topic as well as for a lot of valuable input and many fruitful discussions. Also, I would like to thank the anonymous reviewers for their valuable remarks.

\bibliographystyle{abbrv}

\medskip

\noindent Andreas Nessmann, \href{mailto:andreas.nessmann@tuwien.ac.at}{\texttt{andreas.nessmann@tuwien.ac.at}} 

\noindent Institut für diskrete Mathematik und Geometrie, Technische Universität Wien\\
Institut Denis Poisson, Université de tours

\newpage
\appendix

\section{Proof of Thm.~{\ref{thm:AllHF}}}\label{app:thmAllHF}

\thmAllHF*

\begin{proof}[Outline.] The arguments are mostly the same as in \cite[Thm.~2]{Hung}. From (\ref{eq:FE}), it follows that $K(x,y)H(x,y)$ is already uniquely defined by the (univariate) boundary terms $K(x,0)H(x,0)$ and $K(0,y)H(0,y)$. It suffices to consider two cases:
\begin{enumerate}
    \item $K(0,0)=0$:\\
    In this case, we have $X_+(0)=0$, we can therefore substitute $X_+$ into a power series. Doing so in (\ref{eq:FE}) gives \begin{align}
        0=K(X_+,0)H(X_+,0)+K(0,y)H(0,y).
    \end{align}
    Utilizing this to substitute for $K(0,y)H(0,y)$ in (\ref{eq:FE}), we obtain \begin{align}
        K(x,y)H(x,y)=\underbrace{K(x,0)H(x,0)}_{=:P(x)}-\underbrace{K(X_+,0)H(X_+,0)}_{=:P(X_+)}.
    \end{align}
    Setting \begin{align}\label{eq:defH1_1}
        H_1^m(x,y)=\frac{\omega(x)^m-\omega(X_+)^m}{K(x,y)},
    \end{align}
    and utilizing that around $0$ we have (after scaling and potentially switching $x,y$) $\omega(x)=\frac{x(1+p(x))}{(1-x)^{\pi/\theta}}$ (see \cite[5.3]{Book}; use that our walk is not singular), we can iteratively compute coefficients $a_k$ such that $\sum a_j\omega(x)^k=P(x)$. To see that at the end we indeed obtain a power series, one can apply the Weierstra{\ss} preparation theorem. 
    \item $K(0,0)\neq 0$:\\
    In this case, the previous approach does not work anymore since substitution of $X_+$ into an arbitrary power series fails. Instead, let now $\omega(x)=\sum x^nc^n,\omega\left(X_+\right)=\sum y^n d^n$. We know that $c_1,d_1\neq 0, c_0=0$ (see \cite[5.3]{Book}, and notice that $p_{-1,-1}\neq 0$).\\
    We can now proceed by defining \begin{align}
        P_{2m}(z)&=z^m(z-d_0)^m,\\
        P_{2m+1}(z)&=z^{m+1}(z-d_0)^m.
    \end{align}
    Letting \begin{align}\label{eq:defH1_2}
        H_1^m(x,y):=\frac{P_m\left(\omega(x)\right)-P_m\left(\omega(X_+)\right)}{K(x,y)},
    \end{align}
    one can check that the monomial with non-zero coefficient with minimal degree in the series representation of $H_1^m(x,y)$ around $0$ occurs for $k=l=m$ for $m$ even, and $k=l+1=m$ otherwise. Note here that $\omega(x),\omega(X_0)$ have non-vanishing derivative at $0$ as $0\in\mathcal{G}^\circ$, see \cite[5.3]{Poly}.  From there, given arbitrary power series $Q(x),R(y)$ with $Q(0)=R(0)$, one can again iteratively build coefficients $a_n$ such that $\sum a_nP_n(\omega(x))=Q(x)$, $\sum b_nP_n\left(\omega(X_+)\right)=R(y)$. We have thus constructed a harmonic functions with boundary terms $Q(x),R(y)$; since these were arbitrary we are done. Note that, since $K(0,0)\neq 0$, $1/K(x,y)$ can be written as a power series around $0$.

\end{enumerate}

\end{proof}
\newpage
\section{Proof of Thm.~{\ref{thm:SimpleWalk}}}\label{app:thmSW}
\thmSimpleWalk*
\begin{proof}
For the first part, we use induction over $m$. For $m=1$, the statement is nothing but (\ref{eq:buildHarmonic}). Now, assume the statement holds for $1,\dots, m-1$, and pick an arbitrary $k$. Due to the invariance property (\ref{eq:wInv}), and using that \begin{align}
    xy\frac{\omega(x)-\omega(X_+)}{K(x,y)}=\frac{-8xy}{(1-x)^2(1-y)^2}=2\omega(x)\omega\left(X_+\right),
\end{align}
we see that $X_+yH_1^k(X_+,y)=X_-yH_1^k(X_-,y)$, since we can write $X_+yH_1^k(X_+,y)$ as a polynomial in $\omega,~ \omega(X_+)$. Using (\ref{eq:FE}), we can deduce that a $m$-harmonic function with $\triangle H_m^k=H_{m-1}^k$ is given by \begin{align}
    H_m^k(x,y)=\frac{xyH_{m-1}^k(x,y)-X_+yH_{m-1}^k\left(X_+,y\right)}{K(x,y)}.
\end{align}
After a short computation one obtains 
\begin{multline}
    xyH_m^k(x,y)-X_+yH_m^k\left(X_+,y\right)\\
    =2^{m}\omega\left(X_+\right)^{m}\left[\omega(x)\sum_{j=0}^{k-1}s_{m}(j)\omega\left(X_+\right)^j\omega(x)^{k-1-j}-\omega\left(X_+\right)^k\sum_{j=0}^{k-1}s_{m}(j)\right].
    \end{multline}
Using the algebraic identity \begin{align}
    a\left[\sum_{j=0}^{k-1}c_j a^{k-j-1}b^j\right]-b^n\sum_{j=0}^{k-1}c_j=(a-b)\sum_{j=0}^{k-1}\left(\sum_{i=1}^{j+1}c_i\right)a^{k-j-1}b^j
\end{align}
for $a=\omega(x)$ and $b=\omega\left(X_+\right)$ then yields the statement.\\
For the second part, to see the existence of the limit it suffices to notice that the minimal degree of any non-zero coefficient is at least $m$ (note that $\omega(0)=\omega\left(X_+(0)\right)=0$). In the same fashion we can take the limit on both sides of (\ref{eq:FE}) and see that both sides converge to the same power series. \\
To see that we can in this manner indeed produce all possible polyharmonic functions we proceed as in the proof of Thm.~\ref{thm:AllPHF1}.
\end{proof}
\newpage
\section{Proof of Thm.~{\ref{thm:buildPHF}}}\label{app:thmbuildPHF}
\thmbuildPHF*

\begin{proof}[Outline.] 
Consider first the case $n=1$. $H_1^k(x,y)$ being rational follows immediately from $\pi/\theta\in\mathbb{Z}$, and thus $\omega$ being rational (see \cite[(3.12)]{Conformal}). As by construction the numerator $N_1^k(x,y)$ of $xyH_1^k(x,y)$ as defined in (\ref{eq:buildPHF}) satisfies $N_1^k(X_\pm)=0$, it must be a multiple of $K(x,y)$, thus the only poles of $H_1^k$ can be those coming from $\omega(x),\omega(X_+)$. Since $X_+=1$ only if $y=1$, the statement follows from $(\ref{eq:DecouplingFormula})$. The existence of a decoupling function follows immediately from Prop. \ref{prop:oSum}, and since $X_\pm(1)=Y_\pm(1)=1$, we can conclude that $F_1^k(x)$ has its only pole at $1$. We have thus shown the theorem for $n=1$, except for the computation of $\alpha$ which will be done at the end.\\
Now let $n\geq 2$ and assume the theorem holds for $n-1$. We first want to show that $F_n^k(x)$ (as defined by (\ref{eq:DecouplingFormula}))is a decoupling function. We check the orbit sum criterion (\ref{eq:SignedOrbit}). Using (\ref{eq:buildPHF}), we can utilize that by induction hypothesis we already know that (\ref{eq:SignedOrbit}) is satisfied for $xyH_{n-1}^k(x,y)$; dividing by $\frac{1}{xy}K(x,y)$ does not change this, nor does substituting $X_+$ for $x$ in the numerator. Therefore, it remains to show that $xy\left(F_{n-1}^k(x)-F_{n-1}^k(X_+)\right)/K(x,y)$ admits a decoupling function, but this is an immediate consequence of Cor.~\ref{prop:oSum}. In the same manner as for $n=1$, we conclude that $F_n^k(x)$ has its only pole at $x=1$. For the term $X_+yH_{n+1}^k(X_+,y)-F_{n-1}^k(X_+,y)$, note that it is nothing but the $G(y)$ in (\ref{eq:def:Decoupling}), for which an explicit formula similar to (\ref{eq:DecouplingFormula}) is given in \cite[Th.~4.11]{Tutte}. One can easily see that the arguments for $F_n^k(x)$ can be repeated directly for this $G(y)$, and thus this expression too can have its only pole at $y=1$.\\
$H_n^k(x,y)$ having the form given by (\ref{eq:formPHF}) (for now with any $\alpha$) again follows from the fact that its numerator vanishes at $x=X_\pm$, and thus must contain a factor $K(x,y)$. Lastly, to check that $\triangle H_{n+1}=H_n$ it suffices to substitute into the functional equation (\ref{eq:FE}).\\
It remains to show that the order of the poles at $x,y=1$ is at most $k\cdot \pi/\theta + 2(n-1)$. For $n=1$ this can again be verified directly; afterwards it follows from induction: by a computation one can see that the order of the pole of $F(x)$ compared to the one at $x=1$ of $xyH(x,y)$ increases at most by $2$, and by a similar argument for the $G(y)$ in (\ref{eq:def:Decoupling}) (see \cite[Th.~4.11]{Tutte} for an explicit formula) one can show the same for $X_+yH(X_+,y)-F(X_+)=G(y)$. Using (\ref{eq:buildPHF}) finally yields the statement.
\end{proof}

\end{document}